\documentclass[10pt]{amsart}

\usepackage[a4paper,margin=1.95cm]{geometry}
\usepackage[T1]{fontenc}
\usepackage[utf8]{inputenc}
\usepackage{lmodern}
\usepackage{microtype}
\usepackage{amsmath,amssymb,amsthm,mathtools}
\usepackage[hidelinks,hypertexnames=false]{hyperref}
\usepackage[nameinlink, capitalize, noabbrev]{cleveref}

\newtheorem{theorem}{Theorem}[section]
\newtheorem{proposition}[theorem]{Proposition}
\newtheorem{lemma}[theorem]{Lemma}
\newtheorem{corollary}[theorem]{Corollary}
\numberwithin{equation}{section}

\crefname{theorem}{Theorem}{Theorems}
\crefname{proposition}{Proposition}{Propositions}
\crefname{lemma}{Lemma}{Lemmas}
\crefname{corollary}{Corollary}{Corollaries}
\crefname{equation}{Equation}{Equations}

\newcommand{\bw}{\operatorname{bw}}

\hypersetup{
  pdftitle={A bandwidth refinement of the Erd\H{o}s distinct subset sums bound},
  pdfauthor={Simone Costa; Stefano Della Fiore},
  pdfsubject={Distinct subset sums and hypercube bandwidth},
  pdfkeywords={distinct subset sums, hypercube bandwidth, Catalan numbers}
}

\title[A bandwidth refinement of the Erd\H{o}s distinct subset sums bound]{A bandwidth refinement of the Erd\H{o}s distinct subset sums bound}
\author{Simone Costa}
\address[Simone Costa]{Dipartimento di Ingegneria Civile, Architettura, Territorio, Ambiente e di Matematica (DICATAM), Sezione di Matematica, Universit\`a degli Studi di Brescia, Via Branze 43, 25123 Brescia, Italy}
\author{Stefano Della Fiore}
\address[Stefano Della Fiore]{Dipartimento di Ingegneria dell'Informazione, Universit\`a degli Studi di Brescia, Via Branze 38, 25123 Brescia, Italy}
\date{}

\begin{document}
\begin{abstract}
Let $f(n)$ be the least possible largest element of an $n$-element set of
positive integers with pairwise distinct subset sums.  Dubroff, Fox and Xu
proved the finite lower bound
\[
 f(n)\ge \binom{n}{\lfloor n/2\rfloor}
\]
by applying a vertex-boundary estimate to the $2^{n-1}$ subsets whose sums lie
below half of the total sum.  We instead order all $2^n$ subsets by increasing
subset-sum value.  We prove that the bandwidth of this numbering is at most the largest
element of the set, so the exact formula for the bandwidth of the hypercube
gives the stronger finite bound
\[
 f(n)\ge H_n:=\sum_{j=0}^{n-1}\binom{j}{\lfloor j/2\rfloor}.
\]
An exact identity for $H_n$ in terms of Catalan numbers yields
\[
 \frac{H_n}{\binom{n}{\lfloor n/2\rfloor}}
 =\begin{cases}
  1+\dfrac{2}{3n}+O(n^{-2}),& n\text{ even},\\[5pt]
  1+\dfrac{4}{3n}+O(n^{-2}),& n\text{ odd}.
 \end{cases}
\]
For every $n\ge3$, this bound is strictly larger than the central-binomial
bound.  As a separate corollary, a first-order asymptotic formula for
factorials rewrites this lower bound as a multiple of
$\sqrt{2/\pi}\,2^n/\sqrt n$; the
coefficients of $1/n$ are $5/12$ in even dimension and $7/12$ in odd
dimension.
\end{abstract}
\maketitle
\enlargethispage{2pt}

\section{Introduction}

Let
\[
 A=\{a_1<\cdots<a_n\}\subset \mathbb Z_{>0}
\]
be a set whose $2^n$ subset sums are pairwise distinct.  For
$S\subseteq[n]:=\{1,\ldots,n\}$, write
\[
 w(S):=\sum_{i\in S}a_i.
\]
We set
\[
 f(n):=\min\{a_n: A\subset\mathbb Z_{>0},\ |A|=n,
              \text{$A$ has pairwise distinct subset sums}\}.
\]
Erd\H{o}s asked whether there is an absolute constant $c>0$ such that
\[
 f(n)\ge c2^n.
\]
The problem remains open.  Since all $2^n$ subset sums are distinct integers in
$[0,\sum_i a_i]$ and $\sum_i a_i\le n a_n$, the elementary interval count gives
\[
 a_n\ge \frac{2^n-1}{n}.
\]
The classical Erd\H{o}s--Moser second-moment argument
improves this to
\[
 f(n)\ge \left(\frac14-o(1)\right)\frac{2^n}{\sqrt n};
\]
see Erd\H{o}s~\cite{Erdos1956} and the historical accounts in
\cite{Guy1982,DFX,Steinerberger2023}.  The constant was subsequently improved
by Guy, Elkies, Bae and Aliev
\cite{Guy1982,Elkies1986,Bae1996,Aliev2008}.  The best known leading constant
is $\sqrt{2/\pi}$, first obtained in unpublished work of Elkies and Gleason,
as recorded by Dubroff, Fox and Xu~\cite{DFX}.  They gave two short proofs of
\begin{equation}\label{eq:DFX-asymptotic}
 f(n)\ge
 \left(\sqrt{\frac2\pi}-o(1)\right)\frac{2^n}{\sqrt n}.
\end{equation}
Their isoperimetric proof also gives the finite estimate
\begin{equation}\label{eq:DFX-finite}
 f(n)\ge K_n,
 \qquad
 K_n:=\binom{n}{\lfloor n/2\rfloor}.
\end{equation}
Steinerberger later gave a Fourier-analytic proof of
\cref{eq:DFX-asymptotic}~\cite{Steinerberger2023}.  For upper bounds, Bohman
constructed sum-distinct sets with
\[
 f(n)\le 0.22002 \cdot 2^n
\]
for all sufficiently large $n$~\cite{Bohman1998}.

We now define the graph quantities used below.  The $n$-dimensional hypercube
$Q_n$ is the graph with vertex set $2^{[n]}$, where two subsets are adjacent
when they differ in exactly one element.  If a finite graph $G=(V,E)$ has
$m$ vertices, a \emph{numbering} of $G$ is a bijection
\[
 \eta:V\longrightarrow\{1,\ldots,m\}.
\]
Its bandwidth is
\[
 \bw(\eta):=\max_{uv\in E}|\eta(u)-\eta(v)|,
\]
and the bandwidth of $G$ is
\[
 \bw(G):=\min_{\eta}\bw(\eta).
\]
For a family $\mathcal F\subseteq2^{[n]}$, its \emph{exterior vertex boundary}
is the set of vertices outside $\mathcal F$ that are adjacent in $Q_n$ to at
least one member of $\mathcal F$.

\medskip
\noindent\textbf{Boundary versus bandwidth.}
The two arguments use different graph quantities.  In the isoperimetric proof
of Dubroff, Fox and Xu, one considers the half-cube
\[
 \{S\subseteq[n]:w(S)<\tfrac12 w([n])\}
\]
and estimates its exterior vertex boundary.  This gives the central-binomial
bound $K_n$ in \cref{eq:DFX-finite}.  Here we instead keep the increasing
order of all $2^n$ subset sums $w(S)$.  We prove that the bandwidth of this
numbering is at most $a_n$.  The exact formula in Theorem~\ref{thm:Harper} then gives
a strictly larger finite lower bound for every $n\ge3$.  Thus Dubroff, Fox and
Xu use one exterior vertex boundary, whereas the present argument uses the
bandwidth of the complete subset-sum numbering. A similar use of graph
bandwidth in a different subset-sum problem appears in recent work of
Korsky~\cite{Korsky2026}, where bandwidth is used to study subset sums avoiding
three-term arithmetic progressions.

Theorem~\ref{thm:main} states the exact finite bound and the coefficients of its
relative $1/n$ term.  The quantity $H_n$ is defined in that theorem,
while $K_n$ was defined in \cref{eq:DFX-finite}.

\begin{theorem}\label{thm:main}
For every $n\ge1$,
\begin{equation}\label{eq:main-finite}
 f(n)\ge H_n:=\sum_{j=0}^{n-1}\binom{j}{\lfloor j/2\rfloor}.
\end{equation}
Moreover, $H_n>K_n$ for every $n\ge3$, and, as $n\to\infty$,
\begin{equation}\label{eq:relative-K}
 \frac{H_n}{K_n}
 =\begin{cases}
  1+\dfrac{2}{3n}+O(n^{-2}),&n\text{ even},\\[6pt]
  1+\dfrac{4}{3n}+O(n^{-2}),&n\text{ odd}.
 \end{cases}
\end{equation}
Consequently,
\begin{equation}\label{eq:main-relative-f}
 f(n)\ge K_n
 \begin{cases}
  1+\dfrac{2}{3n}+O(n^{-2}),&n\text{ even},\\[6pt]
  1+\dfrac{4}{3n}+O(n^{-2}),&n\text{ odd}.
 \end{cases}
\end{equation}
\end{theorem}

For odd $n$, the coefficient $4/3$ in \cref{eq:relative-K} has two separate
sources.  The conversion between two adjacent central binomial coefficients
contributes $1/n$, while a finite sum of Catalan numbers contributes
$1/(3n)$.

The paper is organized as follows, Section~\ref{sec:transfer} proves the exact lower bound
$f(n)\ge H_n$ by ordering the vertices of $Q_n$ by their subset sums and then
using Theorem~\ref{thm:Harper}.  Section~\ref{sec:catalan}
rewrites $H_n$ as a central binomial coefficient plus a finite sum of Catalan
numbers and estimates that sum.  Section~\ref{sec:asymptotics} compares $H_n$ with $K_n$ and computes the
coefficients of $1/n$ in Theorem~\ref{thm:main}.  Finally,
Section~\ref{sec:gaussian} uses Theorem~\ref{thm:Stirling} to express the lower bound
as a multiple of $\sqrt{2/\pi}\,2^n/\sqrt n$.

\section{From subset sums to hypercube bandwidth}\label{sec:transfer}

The goal of this section is to prove \cref{eq:main-finite}.  We first state
the exact value of the hypercube bandwidth, and then apply it to the numbering
obtained from the subset sums.

The graph-theoretic input is Theorem~\ref{thm:Harper}, proved by
Harper~\cite{Harper}.  A short recursive proof of the explicit formula was
later given by Wang, Wu and Dumitrescu~\cite{WangWuDumitrescu}.

\begin{theorem}\label{thm:Harper}
For every integer $n\ge1$,
\begin{equation}\label{eq:Harper}
 \bw(Q_n)=\sum_{j=0}^{n-1}\binom{j}{\lfloor j/2\rfloor}.
\end{equation}
\end{theorem}

By Theorem~\ref{thm:Harper}, the right-hand side of \cref{eq:main-finite} is exactly
$\bw(Q_n)$.  Proposition~\ref{prop:transfer} connects this graph quantity with the
largest element $a_n$.

\begin{proposition}\label{prop:transfer}
Let
\[
 0<a_1<\cdots<a_n
\]
be integers whose subset sums are pairwise distinct.  Then
\[
 \bw(Q_n)\le a_n.
\]
Consequently,
\[
 f(n)\ge \bw(Q_n)=H_n.
\]
\end{proposition}

\begin{proof}
Recall from the introduction that
\[
 w(S)=\sum_{i\in S}a_i\qquad(S\subseteq[n]).
\]
The $2^n$ numbers $w(S)$ are pairwise distinct integers.  Order the vertices
of $Q_n$ by increasing $w(S)$ and define
\[
 \eta_A(S):=1+\#\{T\subseteq[n]:w(T)<w(S)\}.
\]
Because the subset sums are distinct, $\eta_A$ is a bijection from
$2^{[n]}$ to $\{1,\ldots,2^n\}$, hence a numbering of $Q_n$.

Consider an edge of $Q_n$.  After exchanging its endpoints if necessary, it
has the form
\[
 S\quad\text{and}\quad S\cup\{i\},
 \qquad i\notin S.
\]
Their subset sums satisfy
\[
 w(S\cup\{i\})-w(S)=a_i>0,
\]
so $\eta_A(S\cup\{i\})>\eta_A(S)$.  Set
\[
 d:=\eta_A(S\cup\{i\})-\eta_A(S).
\]
Exactly $d-1$ vertices occur strictly between the two endpoints in the
numbering.  Since the numbering is increasing in $w$, their subset sums are
$d-1$ distinct integers strictly between the two integers $w(S)$ and
$w(S)+a_i$.  The open interval
\[
 (w(S),w(S)+a_i)
\]
contains exactly the $a_i-1$ integers
\[
 w(S)+1,\ldots,w(S)+a_i-1.
\]
Therefore
\[
 d-1\le a_i-1,
\]
and hence
\[
 |\eta_A(S\cup\{i\})-\eta_A(S)|=d\le a_i\le a_n.
\]
This estimate holds for every edge of $Q_n$.  By the definition of the
bandwidth of a numbering,
\[
 \bw(\eta_A)\le a_n.
\]
Since $\bw(Q_n)$ is the minimum of $\bw(\eta)$ over all numberings $\eta$ of
$Q_n$,
\[
 \bw(Q_n)\le \bw(\eta_A)\le a_n.
\]
Finally, minimizing $a_n$ over all $n$-element sets with pairwise distinct
subset sums gives
\[
 f(n)\ge\bw(Q_n).
\]
Combining this with Theorem~\ref{thm:Harper} proves \cref{eq:main-finite}.
\end{proof}

\section{Exact formulas for \texorpdfstring{$H_n$}{H n} and a Catalan sum}\label{sec:catalan}

The goal of this section is to express $H_n$ as a central binomial coefficient
plus a finite sum of Catalan numbers, and then to estimate that Catalan sum.

For $j\ge0$, define the $j$th Catalan number by
\[
 C_j:=\frac{1}{j+1}\binom{2j}{j}.
\]
For $r\ge1$, put
\[
 \Gamma_r:=\sum_{j=1}^{r-1}C_j,
\]
with the convention $\Gamma_1=0$.

Lemma~\ref{lem:Catalan-decomposition} rewrites $H_n$ in terms of one central
binomial coefficient and the Catalan sum $\Gamma_r$.

\begin{lemma}\label{lem:Catalan-decomposition}
For every $r\ge1$,
\begin{equation}\label{eq:Catalan-decomposition}
 H_{2r}=\binom{2r}{r}+\Gamma_r,
 \qquad
 H_{2r+1}=2\binom{2r}{r}+\Gamma_r.
\end{equation}
\end{lemma}

\begin{proof}
For $j\ge0$, write
\[
 B_j:=\binom{2j}{j}.
\]
By the definition of $H_n$ in Theorem~\ref{thm:main},
\[
 H_{2r}=\sum_{m=0}^{2r-1}\binom{m}{\lfloor m/2\rfloor}.
\]
Pair the terms with indices $m=2j$ and $m=2j+1$.  This gives
\begin{equation}\label{eq:paired-Harper}
 H_{2r}=\sum_{j=0}^{r-1}
 \left(B_j+\binom{2j+1}{j}\right).
\end{equation}
We now rewrite one pair.  The elementary identities
\[
 \binom{2j+1}{j}=\frac{2j+1}{j+1}B_j,
 \qquad
 B_{j+1}=\frac{2(2j+1)}{j+1}B_j,
 \qquad
 C_j=\frac{1}{j+1}B_j
\]
show that
\begin{align*}
 B_{j+1}-B_j+C_j
 =\left(\frac{2(2j+1)}{j+1}-1+\frac{1}{j+1}\right)B_j
 =\frac{3j+2}{j+1}B_j.
\end{align*}
On the other hand,
\begin{align*}
 B_j+\binom{2j+1}{j}
 =\left(1+\frac{2j+1}{j+1}\right)B_j
 =\frac{3j+2}{j+1}B_j.
\end{align*}
Thus, for every $j\ge0$,
\begin{equation}\label{eq:pair-telescope}
 B_j+\binom{2j+1}{j}=B_{j+1}-B_j+C_j.
\end{equation}
Substituting \cref{eq:pair-telescope} into \cref{eq:paired-Harper} yields
\begin{align*}
 H_{2r}
 =\sum_{j=0}^{r-1}(B_{j+1}-B_j)
   +\sum_{j=0}^{r-1}C_j
 =B_r-B_0+C_0+\sum_{j=1}^{r-1}C_j.
\end{align*}
Since $B_0=C_0=1$, the two constant terms cancel.  Hence
\[
 H_{2r}=B_r+\Gamma_r
       =\binom{2r}{r}+\Gamma_r.
\]
Finally, the only extra summand in $H_{2r+1}$ is the term with index $2r$:
\[
 H_{2r+1}=H_{2r}+\binom{2r}{r}.
\]
Using the formula just proved for $H_{2r}$ gives
\[
 H_{2r+1}=2\binom{2r}{r}+\Gamma_r.
\]
\end{proof}

Lemma~\ref{lem:Catalan-tail} estimates the Catalan sum to the precision needed for
a relative error of order $n^{-2}$ in Theorem~\ref{thm:main}.

\begin{lemma}\label{lem:Catalan-tail}
As $r\to\infty$,
\begin{equation}\label{eq:Catalan-tail}
 \Gamma_r=\left(\frac43+O(r^{-1})\right)C_{r-1}.
\end{equation}
\end{lemma}

\begin{proof}
For $r\ge2$, define
\[
 q_r:=\frac{\Gamma_r}{C_{r-1}}.
\]
Since
\[
 \Gamma_{r+1}=\Gamma_r+C_r,
\]
we have
\begin{align*}
 q_{r+1}
 &=\frac{\Gamma_{r+1}}{C_r}
 =\frac{\Gamma_r+C_r}{C_r}
 =1+\frac{C_{r-1}}{C_r}\frac{\Gamma_r}{C_{r-1}}=1+\alpha_r q_r,
\end{align*}
where
\begin{align}
 \alpha_r
 :=\frac{C_{r-1}}{C_r}
 =\frac{\frac1r\binom{2r-2}{r-1}}
         {\frac1{r+1}\binom{2r}{r}}
 =\frac{r+1}{r}\frac{r^2}{(2r)(2r-1)}
 =\frac{r+1}{2(2r-1)}.
 \label{eq:alpha-r}
\end{align}
Thus
\begin{equation}\label{eq:q-rec}
 q_{r+1}=1+\alpha_rq_r.
\end{equation}

Write
\[
 q_r=\frac43+e_r.
\]
Substituting this into \cref{eq:q-rec} gives
\begin{align*}
 e_{r+1}
 =1+\alpha_r\left(\frac43+e_r\right)-\frac43
 =\alpha_re_r+\left(\frac{4\alpha_r}{3}-\frac13\right).
\end{align*}
Using \cref{eq:alpha-r},
\begin{align*}
 \frac{4\alpha_r}{3}-\frac13
 =\frac{2(r+1)}{3(2r-1)}-\frac13
 =\frac{2r+2-(2r-1)}{3(2r-1)}
 =\frac{1}{2r-1}.
\end{align*}
Therefore
\begin{equation}\label{eq:e-rec}
 e_{r+1}=\alpha_re_r+\frac{1}{2r-1}.
\end{equation}
For $r\ge2$,
\[
 0<\alpha_r=\frac{r+1}{2(2r-1)}\le\frac12.
\]
We now prove by induction that
\begin{equation}\label{eq:e-bound}
 |e_r|\le\frac4r\qquad(r\ge2).
\end{equation}
For $r=2$,
\[
 \Gamma_2=C_1=1,
 \qquad
 q_2=1,
 \qquad
 e_2=-\frac13,
\]
so \cref{eq:e-bound} holds.  Assume it holds for some $r\ge2$.  Then
\cref{eq:e-rec} and $\alpha_r\le1/2$ give
\[
 |e_{r+1}|
 \le\frac12\frac4r+\frac{1}{2r-1}
 =\frac2r+\frac{1}{2r-1}.
\]
It remains to compare this with $4/(r+1)$.  After subtracting $2/r$ from
both sides, the required inequality is
\[
 \frac{1}{2r-1}\le\frac4{r+1}-\frac2r
 =\frac{2(r-1)}{r(r+1)}.
\]
Multiplying by the positive quantity $r(r+1)(2r-1)$, this is equivalent to
\[
 r(r+1)\le2(r-1)(2r-1).
\]
The difference between the right- and left-hand sides is
\[
 2(r-1)(2r-1)-r(r+1)
 =3r^2-7r+2
 =(3r-1)(r-2),
\]
which is nonnegative for $r\ge2$.  Hence
\[
 |e_{r+1}|\le\frac4{r+1},
\]
and the induction is complete.  Therefore $e_r=O(r^{-1})$, and since
$q_r=4/3+e_r$,
\[
 \Gamma_r=q_rC_{r-1}
 =\left(\frac43+O(r^{-1})\right)C_{r-1}.
\]
\end{proof}

\section{Comparison with \texorpdfstring{$K_n$}{K n} and the coefficients of \texorpdfstring{$1/n$}{1/n}}\label{sec:asymptotics}

The goal of this section is to prove the remaining assertions of
Theorem~\ref{thm:main}.  We first compare $H_n$ exactly with the central-binomial
bound $K_n$ from \cref{eq:DFX-finite}, and then use
Lemma~\ref{lem:Catalan-tail} to obtain the coefficients $2/3$ and $4/3$.

Proposition~\ref{prop:finite-refinement} gives the exact comparison between $H_n$ and
$K_n$.  It also shows that the bandwidth bound is strictly stronger than
$K_n$ for every $n\ge3$.

\begin{proposition}\label{prop:finite-refinement}
For every $r\ge1$,
\begin{align}
 H_{2r}-K_{2r}&=\Gamma_r,
 \label{eq:even-refinement}\\
 H_{2r+1}-K_{2r+1}
 &=\frac{K_{2r+1}}{2r+1}+\Gamma_r.
 \label{eq:odd-refinement}
\end{align}
Consequently,
\[
 H_n>K_n\qquad(n\ge3).
\]
\end{proposition}

\begin{proof}
For even dimension,
\[
 K_{2r}=\binom{2r}{r}.
\]
The first identity in Lemma~\ref{lem:Catalan-decomposition} is
\[
 H_{2r}=\binom{2r}{r}+\Gamma_r,
\]
so subtracting $K_{2r}$ gives
\[
 H_{2r}-K_{2r}=\Gamma_r.
\]

For odd dimension, put
\[
 N:=\binom{2r}{r}.
\]
The adjacent binomial-coefficient identity gives
\[
 K_{2r+1}=\binom{2r+1}{r}
 =\frac{2r+1}{r+1}N.
\]
Solving for $N$,
\[
 N=\frac{r+1}{2r+1}K_{2r+1}.
\]
Therefore
\begin{align*}
 2N
 =\frac{2r+2}{2r+1}K_{2r+1}
 =\left(1+\frac{1}{2r+1}\right)K_{2r+1}
 =K_{2r+1}+\frac{K_{2r+1}}{2r+1}.
\end{align*}
The second identity in Lemma~\ref{lem:Catalan-decomposition} says
\[
 H_{2r+1}=2N+\Gamma_r.
\]
Substituting the previous expression for $2N$ yields
\[
 H_{2r+1}
 =K_{2r+1}+\frac{K_{2r+1}}{2r+1}+\Gamma_r,
\]
which is \cref{eq:odd-refinement}.

If $n=2r+1\ge3$, the term $K_{2r+1}/(2r+1)$ is positive.  If
$n=2r\ge4$, then $r\ge2$ and $\Gamma_r$ contains the positive term
$C_1=1$.  Hence $H_n>K_n$ for every $n\ge3$.
\end{proof}

We now compute the coefficients of $1/n$ in Theorem~\ref{thm:main}.  Every
conversion between $H_n$, $K_n$ and $C_{r-1}$ is written explicitly.

\begin{proof}[Proof of the asymptotic assertions in Theorem~\ref{thm:main}]
\medskip
\noindent\emph{Even dimensions.}
Let $n=2r$.  By Lemma~\ref{lem:Catalan-decomposition},
\[
 H_n=K_n+\Gamma_r,
 \qquad
 K_n=\binom{2r}{r}.
\]
From the definition of $C_{r-1}$,
\begin{align*}
 C_{r-1}
 =\frac1r\binom{2r-2}{r-1}
 =\frac1r\frac{r^2}{(2r)(2r-1)}\binom{2r}{r}
 =\frac{K_n}{2(2r-1)}
 =\frac{K_n}{2(n-1)}.
\end{align*}
By Lemma~\ref{lem:Catalan-tail}, and since $r^{-1}=2n^{-1}$,
\begin{align*}
 \frac{\Gamma_r}{K_n}
 =\left(\frac43+O(n^{-1})\right)\frac{1}{2(n-1)}
 =\frac{2}{3(n-1)}+O(n^{-2}).
\end{align*}
Moreover,
\[
 \frac{2}{3(n-1)}-\frac{2}{3n}
 =\frac{2}{3n(n-1)}=O(n^{-2}),
\]
so
\[
 \frac{\Gamma_r}{K_n}=\frac{2}{3n}+O(n^{-2}).
\]
Therefore
\[
 \frac{H_n}{K_n}=1+\frac{2}{3n}+O(n^{-2}).
\]

\medskip
\noindent\emph{Odd dimensions.}
Let $n=2r+1$ and put
$
 N:=\binom{2r}{r}.
$
Then
\[
 K_n=\binom{2r+1}{r}=\frac{2r+1}{r+1}N,
\]
so
\begin{align}
 \frac{2N}{K_n}
 =\frac{2(r+1)}{2r+1}
 =\frac{n+1}{n}
 =1+\frac1n.
 \label{eq:odd-central}
\end{align}
Also,
\begin{align*}
 C_{r-1}
 =\frac1r\binom{2r-2}{r-1}
 =\frac{N}{2(2r-1)}.
\end{align*}
Therefore
\begin{align}
 \frac{C_{r-1}}{K_n}
 =\frac{N}{2(2r-1)}\frac{r+1}{(2r+1)N}
 =\frac{r+1}{2(2r-1)(2r+1)}
 =\frac{n+1}{4n(n-2)}.
 \label{eq:odd-Catalan-normalization}
\end{align}
By Lemma~\ref{lem:Catalan-tail}, with $r^{-1}=O(n^{-1})$,
\begin{align*}
 \frac{\Gamma_r}{K_n}
 =\left(\frac43+O(n^{-1})\right)
   \frac{n+1}{4n(n-2)}
 =\frac{n+1}{3n(n-2)}+O(n^{-2}).
\end{align*}
Since
\[
 \frac{n+1}{3n(n-2)}-\frac1{3n}
 =\frac{1}{n(n-2)}=O(n^{-2}),
\]
we obtain
\[
 \frac{\Gamma_r}{K_n}=\frac1{3n}+O(n^{-2}).
\]
Finally, Lemma~\ref{lem:Catalan-decomposition} gives $H_n=2N+\Gamma_r$, so
\cref{eq:odd-central} yields
\[
 \frac{H_n}{K_n}
 =1+\frac1n+\frac1{3n}+O(n^{-2})
 =1+\frac{4}{3n}+O(n^{-2}).
\]
This proves \cref{eq:relative-K}; combining it with $f(n)\ge H_n$ gives
\cref{eq:main-relative-f} and completes the proof of Theorem~\ref{thm:main}.
\end{proof}

\section{Expansion relative to \texorpdfstring{$\sqrt{2/\pi}\,2^n/\sqrt n$}{sqrt(2/pi) times 2 to the n over sqrt(n)}}\label{sec:gaussian}

The goal of this section is to derive the coefficients $5/12$ and $7/12$ when
the lower bound is written as a multiple of
$\sqrt{2/\pi}\,2^n/\sqrt n$, as in \cref{eq:DFX-asymptotic}.

The factorial asymptotic used below is stated in Theorem~\ref{thm:Stirling}; see
Robbins~\cite{Robbins1955}.

\begin{theorem}\label{thm:Stirling}
As $m\to\infty$ through positive integers,
\begin{equation}\label{eq:Stirling}
 m!=\sqrt{2\pi m}\left(\frac me\right)^m
 \left(1+\frac{1}{12m}+O(m^{-2})\right).
\end{equation}
\end{theorem}

Lemma~\ref{lem:Kn-Stirling} gives the asymptotic expansion of $K_n$ for even and
odd $n$.

\begin{lemma}\label{lem:Kn-Stirling}
As $n\to\infty$,
\begin{equation}\label{eq:Kn-Stirling}
 K_n=\sqrt{\frac2\pi}\frac{2^n}{\sqrt n}
 \begin{cases}
  1-\dfrac{1}{4n}+O(n^{-2}),&n\text{ even},\\[6pt]
  1-\dfrac{3}{4n}+O(n^{-2}),&n\text{ odd}.
 \end{cases}
\end{equation}
\end{lemma}

\begin{proof}
Let first $n=2r$.  Applying Theorem~\ref{thm:Stirling} to $(2r)!$ and $r!$ gives
\[
 \binom{2r}{r}
 =\frac{4^r}{\sqrt{\pi r}}
 \frac{1+\frac{1}{24r}+O(r^{-2})}
      {1+\frac{1}{6r}+O(r^{-2})}.
\]
Since $(1+x)^{-1}=1-x+O(x^2)$ for $x=O(r^{-1})$,
\[
 \frac{1}{1+\frac{1}{6r}+O(r^{-2})}
 =1-\frac{1}{6r}+O(r^{-2}),
\]
and hence
\[
 \binom{2r}{r}
 =\frac{4^r}{\sqrt{\pi r}}
 \left(1-\frac{1}{8r}+O(r^{-2})\right).
\]
Using $n=2r$, $4^r=2^n$ and $\sqrt r=\sqrt{n/2}$ gives the even case of
\cref{eq:Kn-Stirling}.

Now let $n$ be odd.  The adjacent-binomial identity gives
\[
 K_n=\frac{2n}{n+1}K_{n-1}.
\]
Since $n-1$ is even, the case just proved gives
\[
 K_{n-1}
 =\sqrt{\frac2\pi}\frac{2^{n-1}}{\sqrt{n-1}}
 \left(1-\frac{1}{4(n-1)}+O(n^{-2})\right).
\]
Therefore
\[
 K_n
 =\sqrt{\frac2\pi}\frac{2^n}{\sqrt n}
 \left(\frac{n}{n+1}\sqrt{\frac{n}{n-1}}\right)
 \left(1-\frac{1}{4(n-1)}+O(n^{-2})\right).
\]
Using
\[
 \frac{n}{n+1}=1-\frac1n+O(n^{-2}),
 \qquad
 \sqrt{\frac{n}{n-1}}=1+\frac{1}{2n}+O(n^{-2}),
\]
the first parenthesis equals $1-1/(2n)+O(n^{-2})$, while the second equals
$1-1/(4n)+O(n^{-2})$.  Their product is
$1-3/(4n)+O(n^{-2})$, proving the odd case.
\end{proof}

Combining Lemma~\ref{lem:Kn-Stirling} with Theorem~\ref{thm:main} gives
Corollary~\ref{cor:Gaussian}.

\begin{corollary}\label{cor:Gaussian}
As $n\to\infty$,
\begin{equation}\label{eq:relative-Gaussian}
 f(n)\ge\sqrt{\frac2\pi}\frac{2^n}{\sqrt n}
 \begin{cases}
  1+\dfrac{5}{12n}+O(n^{-2}),&n\text{ even},\\[6pt]
  1+\dfrac{7}{12n}+O(n^{-2}),&n\text{ odd}.
 \end{cases}
\end{equation}
\end{corollary}

\begin{proof}
For even $n$, multiply the even cases of \cref{eq:main-relative-f} and
\cref{eq:Kn-Stirling}.  The coefficient of $1/n$ is
\[
 -\frac14+\frac23=\frac5{12}.
\]
For odd $n$, the coefficient of $1/n$ is
\[
 -\frac34+\frac43=\frac7{12}.
\]
In both cases the product of the two $O(n^{-1})$ corrections contributes only
$O(n^{-2})$.  This proves \cref{eq:relative-Gaussian}.
\end{proof}

\section*{Declaration of Generative AI} During the preparation of this work, the author used ChatGPT~5.6 to improve the language, readability, and clarity of the exposition. After using this tool, the author carefully reviewed and edited the manuscript as needed. The author takes full responsibility for the content of the article.

{

}

\end{document}